\documentclass[11pt,a4paper]{amsart}
\usepackage{amsfonts,amsmath,amssymb,amsthm,mathtools}
\usepackage[noadjust]{cite}

\numberwithin{equation}{section}

\DeclareMathOperator{\Ran}{Ran}
\DeclareMathOperator{\spec}{spec}
\DeclareMathOperator{\sign}{sign}
\DeclareMathOperator{\dist}{dist}
\DeclareMathOperator{\conv}{conv}

\DeclareSymbolFont{SY}{U}{psy}{m}{n}
\DeclareMathSymbol{\emptyset}{\mathord}{SY}{'306}

\DeclarePairedDelimiter{\abs}{|}{|}
\DeclarePairedDelimiter{\norm}{\lVert}{\rVert}

\newcommand{\dd}{\mathrm d}

\newcommand{\NN}{\mathbb{N}}
\newcommand{\RR}{\mathbb{R}}
\newcommand{\EE}{\mathsf{E}}

\theoremstyle{plain}
\newtheorem{theorem}{Theorem}[section]
\newtheorem{proposition}[theorem]{Proposition}
\newtheorem{lemma}[theorem]{Lemma}
\newtheorem{corollary}[theorem]{Corollary}

\theoremstyle{remark}
\newtheorem{remark}[theorem]{Remark}

\title[Unifying perturbations in the subspace perturbation problem]{Unifying the treatment of indefinite and semidefinite
perturbations in the subspace perturbation problem}
\subjclass[2010]{Primary 47A55; Secondary 47A15, 47B15}
\keywords{Subspace perturbation problem, spectral subspaces, maximal angle between closed subspaces}
\date{}

\author[A.\ Seelmann]{Albrecht Seelmann}
\address{A.~Seelmann,
 Technische Univer\-si\-t\"at Dortmund, Fakult\"at f\"ur Mathematik, D-44221 Dortmund, Germany}
\email{albrecht.seelmann@math.tu-dortmund.de}

\begin{document}

\begin{abstract}
  The variation of spectral subspaces for linear self-adjoint operators under an additive bounded perturbation is considered. The
  objective is to estimate the norm of the difference of two spectral projections associated with isolated parts of the spectrum
  of the perturbed and unperturbed operators. Recent results for semidefinite and general, not necessarily semidefinite,
  perturbations are unified to statements that cover both types of perturbations and, at the same time, also allow for certain
  perturbations that were not covered before.
\end{abstract}

\maketitle

\section{Introduction and main result}

The present note continues the considerations on the~\emph{subspace perturbation problem} previously discussed in several recent
works such as~\cite{AM13,Seel18,Seel16,Seel19}; see also the references cited therein. More specifically, the results for
semidefinite perturbations from~\cite{Seel19} and those for general, not necessarily semidefinite, perturbations
from~\cite{Seel18,Seel14} are unified to general statements which cover both types of perturbations and, at the same time, allow
for certain perturbations that have not been covered before. This is achieved by decomposing the perturbation into its
nonnegative and nonpositive parts with respect to its spectral decomposition. Naturally, the corresponding results here bare a
great similarity to its predecessors. Since the proofs require only small modifications to the previous ones and the essential
parts of the theory remain the same, this note concentrates on giving the formal statements but otherwise skips on discussions
as well as details of the proofs as much as possible and refers to the previous works instead.

Let $A$ be a self-adjoint, not necessarily bounded, operator on a separable Hilbert space such that the spectrum of $A$ is
separated into two disjoint components, that is,
\begin{equation}\label{eq:specSep}
  \spec(A)
  =
  \sigma \cup \Sigma
  \quad\text{ with }\
  d
  :=
  \dist(\sigma, \Sigma)
  >
  0
  .
\end{equation}
Moreover, given a bounded self-adjoint operator $V$ on the same Hilbert space, define bounded nonnegative operators $V_\pm$ with
$V = V_+ - V_-$ via functional calculus by
\begin{equation*}
  V_+ := \bigl( 1 + \sign(V) \bigr)V / 2
  ,\quad
  V_- := \bigl( \sign(V) - 1 \bigr)V / 2
  .
\end{equation*}
We clearly have $\norm{V_\pm} \le \norm{V}$, and $V_-$ or $V_+$ vanish if $V$ is nonnegative or nonpositive, respectively.

If
\begin{equation}\label{eq:normBound}
  \norm{V_+} + \norm{V_-}
  <
  d
  ,
\end{equation}
then it can be shown (see Corollary~\ref{cor:specPert} below and the discussion thereafter) that the spectrum of the perturbed
operator $A+V$ is likewise separated into two disjoint components,
\begin{equation}\label{eq:specSepPert}
  \spec(A+V)
  =
  \omega \cup \Omega
  \quad\text{ with }\
  \dist(\omega,\Omega)
  \ge
  d - \norm{V_+} - \norm{V_-}
  ,
\end{equation}
where
\begin{equation}\label{eq:defomega}
  \omega
  =
  \spec(A+V) \cap \bigl( \sigma + [-\norm{V_-},\norm{V_+}] \bigr)
\end{equation}
and analogously for $\Omega$ (with $\sigma$ replaced by $\Sigma$); here we have used the short-hand notation
$\sigma + [-\norm{V_-},\norm{V_+}] := \{ \lambda + t \colon \lambda \in \sigma,\ -\norm{V_-} \le t \le \norm{V_+} \}$. Clearly,
the~\emph{gap non-closing} condition~\eqref{eq:normBound} is sharp. Also note that~\eqref{eq:normBound} covers semidefinite
perturbations $V$ with $\norm{V} < d$,  as well as general, not necessarily semidefinite, perturbations satisfying
$\norm{V} < d/2$. On the other hand, condition~\eqref{eq:normBound} also includes certain indefinite perturbations with
$d / 2 \le \norm{V} < d$ that were not covered in the previous works. It is interesting to observe that~\eqref{eq:normBound}
formally differs from the condition $\norm{V} < d$ for semidefinite perturbations and $\norm{V} < d/2$ for general perturbations
by the appearance of $\norm{V_+} + \norm{V_-}$ instead of $\norm{V}$ and $2\norm{V}$, respectively. In fact, this seemingly
na\"ive observation remains valid also when it comes to the main results discussed here and, in this sense, represents the
essence of the present note.

The variation of the spectral subspaces associated with the components of the spectrum is studied in terms of the difference of
the corresponding spectral projections $\EE_A(\sigma)$ and $\EE_{A+V}(\omega)$, where $\EE_A$ and $\EE_{A+V}$ denote the
projection-valued spectral measures for the unperturbed and perturbed self-adjoint operators $A$ and $A+V$, respectively.

The first main result discussed here holds under a certain favourable spectral separation condition for the unperturbed operator
$A$ and unifies~\cite[Theorem~1.1]{Seel19} for semidefinite perturbations and its corresponding well-known variant for general
perturbations (cf., e.g.,~\cite[Remark~2.9]{Seel14}).

\begin{theorem}\label{thm:favGeom}
  Let $A$ be a self-adjoint operator on a separable Hilbert space such that the spectrum of $A$ is separated as
  in~\eqref{eq:specSep}. Let $V$ be a bounded self-adjoint operator on the same Hilbert space satisfying~\eqref{eq:normBound},
  and choose $\omega \subset \spec(A+V)$ as in~\eqref{eq:defomega}.

  If, addition, the convex hull of one of the components $\sigma$ and $\Sigma$ is disjoint from the other component, that is,
  $\conv(\sigma) \cap \Sigma = \emptyset$ or $\sigma \cap \conv(\Sigma) = \emptyset$, then
  \begin{equation}\label{eq:maxAnglefavGeom}
    \arcsin\bigl( \norm{\EE_A(\sigma) - \EE_{A+V}(\omega)} \bigr)
    \le
    \frac{1}{2} \arcsin\Bigl( \frac{\norm{V_+} + \norm{V_-}}{d} \Bigr)
    <
    \frac{\pi}{4}
    ,
  \end{equation}
  and this estimate is sharp.
\end{theorem}

As indicated above, Theorem~\ref{thm:favGeom} differs from its predecessors by the appearance of $\norm{V_+} + \norm{V_-}$
instead of $\norm{V}$ and $2\norm{V}$, respectively. The same difference shows up when considering the generic result where no
additional spectral separation condition other than~\eqref{eq:specSep} is assumed. In order to discuss this here, it is necessary
to recall from~\cite{Seel18} (cf.~also~\cite{Seel19}) the function that has played the crucial role in the corresponding results
for general and semidefinite perturbations:

Set
\begin{equation*}
  c_\text{crit}
  :=
  \frac{1}{2} - \frac{1}{2} \Bigl( 1 - \frac{\sqrt{3}}{\pi} \Bigr)^3
  =
  0{.}4548399\ldots
\end{equation*}
and define $N\colon [0,c_\text{crit}] \to [0,\pi/2]$ by
\begin{equation}\label{eq:boundFunc}
  N(x)
  =
  \begin{cases}
    \frac{1}{2}\arcsin(\pi x) & \text{ for }\quad 0\le x\le \frac{4}{\pi^2+4},\\[0.15cm]
    \arcsin\Bigl(\sqrt{\frac{2\pi^2x-4}{\pi^2-4}}\,\Bigr) & \text{ for }\quad \frac{4}{\pi^2+4} < x < 4\frac{\pi^2-2}
      {\pi^4},\\[0.15cm]
    \arcsin\bigl(\frac{\pi}{2}(1-\sqrt{1-2x}\,)\bigr) & \text{ for }\quad 4\frac{\pi^2-2}{\pi^4} \le x \le \kappa,\\[0.15cm]
    \frac{3}{2}\arcsin\bigl(\frac{\pi}{2}(1-\sqrt[\leftroot{4}3]{1-2x}\,)\bigr) & \text{ for }\quad\kappa < x \le c_\text{crit}.
  \end{cases}
\end{equation}
Here, $\kappa \in ( 4\frac{\pi^2-2}{\pi^4} , 2\frac{\pi-1}{\pi^2} )$ is the unique solution to the equation
\begin{equation*}
  \arcsin\Bigl(\frac{\pi}{2}\bigl(1-\sqrt{1-2\kappa}\,\bigr)\Bigr)=\frac{3}{2}\arcsin\Bigl(\frac{\pi}{2}
  \bigl(1-\sqrt[\leftroot{4}3]{1-2\kappa}\,\bigr)\Bigr)
\end{equation*}
in the interval $(0,2\frac{\pi-1}{\pi^2}]$.

The second principal result of this note now unifies~\cite[Theorem~1.2]{Seel19} for semidefinite perturbations
and~\cite[Theorem~1]{Seel18} for general perturbations.

\begin{theorem}\label{thm:gen}
  Let $A$ be a self-adjoint operator on a separable Hilbert space such that the spectrum of $A$ is separated as
  in~\eqref{eq:specSep}, and let $V$ and $\omega$ as in Theorem~\ref{thm:favGeom}. If, in addition, $V$ satisfies
  \begin{equation*}
    \norm{V_+} + \norm{V_-}
    <
    2c_\text{crit} \cdot d
    ,
  \end{equation*}
  then
  \begin{equation*}
    \arcsin\bigl(\norm{\EE_A({\sigma})-\EE_{A+V}(\omega))}\bigr)
    \le
    N\Bigl(\frac{\norm{V_+}+\norm{V_-}}{2d}\Bigr)
    <
    \frac{\pi}{2}
    ,
  \end{equation*}
  where $N$ is given by~\eqref{eq:boundFunc}.
\end{theorem}

A more detailed discussion on the function $N$ can be found in~\cite{Seel18}.

The rest of this note is organized as follows:

In Section~\ref{subsec:specPert}, the perturbation of the spectrum with respect to the decomposition $V = V_+ - V_-$ of the
perturbation and, in particular, the spectral separation~\eqref{eq:specSepPert},~\eqref{eq:defomega} is discussed.
Section~\ref{subsec:sin2Theta} is devoted to a corresponding variant of the Davis-Kahan $\sin2\Theta$ theorem, which is the core
of the proofs of the main theorems. The latter are finally presented in Section~\ref{subsec:mainproofs}.

\section{Proofs}

\subsection{Perturbation of the spectrum}\label{subsec:specPert}

The following result extends the statement of~\cite[Proposition~2.1]{Seel19} to not necessarily semidefinite perturbations;
cf.~also~\cite[Theorem~3.2]{Ves08} and~\cite[Eq.~(9.4.4)]{BS87}.

\begin{proposition}\label{prop:specPert}
  Let the finite interval $(a,b) \subset \RR$, $a < b$, be contained in the resolvent set of the self-adjoint operator $A$.
  Moreover, let $V$ be a bounded self-adjoint operator on the same Hilbert space with $\norm{V_+} + \norm{V_-} < b-a$. Then, the
  interval $( a+\norm{V_+} , b-\norm{V_-} )$ belongs to the resolvent set of the perturbed operator $A+V$.
\end{proposition}

\begin{proof}
  Since $V_+$ is nonnegative with $\norm{V_+} < b-a$, it follows from~\cite[Proposition~2.1]{Seel19} that the interval
  $( a+\norm{V_+} , b )$ belongs to the resolvent set of $A+V_+$. In turn, since $-V_-$ is nonpositive and
  $\norm{V_-} < b - a -\norm{V_+}$, it follows from the same result that the interval $( a+\norm{V_+} , b-\norm{V_-} )$ belongs
  to the resolvent set of the operator $A+V = (A+V_+) - V_-$.
\end{proof}%

We have the following corollary to Proposition~\ref{prop:specPert}. The proof works exactly as the one
of~\cite[Corollary~2.2]{Seel19} and is hence omitted.

\begin{corollary}\label{cor:specPert}
  Let $A$ be a self-adjoint operator, and let $V$ be a bounded self-adjoint operator on the same Hilbert space. Then,
  \begin{equation*}
    \spec(A+V)
    \subset
    \spec(A) + [ -\norm{V_-} , \norm{V_+} ]
    .
  \end{equation*}
\end{corollary}

It is easy to see from Corollary~\ref{cor:specPert} that in the situation of Theorems~\ref{thm:favGeom} and~\ref{thm:gen} the
spectrum of the perturbed operator is indeed separated as in~\eqref{eq:specSepPert} and~\eqref{eq:defomega}. Moreover, for each
$t \in [0,1]$ we have $(tV)_\pm = tV_\pm$, so that the spectrum of $A+tV$ is likewise separated into two disjoint components
$\omega_t$ and $\Omega_t$, defined analogously to $\omega$ and $\Omega$ in~\eqref{eq:defomega}, respectively. Namely,
\begin{equation}\label{eq:defomegat}
  \omega_t
  =
  \spec(A+tV) \cap \bigl( \sigma + [-t\norm{V_-} , t\norm{V_+}] \bigr)
\end{equation}
and analogously for $\Omega_t$ (with $\sigma$ replaced by $\Sigma$).

We need the following variant of~\cite[Lemma~2.3]{Seel19} (see~also~\cite[Theorem~3.5]{AM13}) for future reference.
\begin{lemma}\label{lem:projCont}
  Let $A$ be as in Theorem~\ref{thm:gen}, and let $V$ be a bounded self-adjoint operator on the same Hilbert space satisfying
  $\norm{V_+} + \norm{V_-} < d$. Then, the projection-valued path $[0,1] \ni t \mapsto \EE_{A+tV}(\omega_t)$ with $\omega_t$ as
  in~\eqref{eq:defomegat} is continuous in norm.
\end{lemma}

\begin{proof}
  Let $0 \le s \le t \le 1$. Since $\dist(\omega_s,\Omega_t) \ge d - t\norm{V_+} - t\norm{V_-}$ as well as
  $\dist(\Omega_s,\omega_t) \ge d - t\norm{V_+} - t\norm{V_-}$, we obtain as in the proof of~\cite[Lemma~2.3]{Seel19}
  (see~also~\cite[Theorem~3.5]{AM13}) that
  \begin{equation*}
    \norm{ \EE_{A+sV}(\omega_s) - \EE_{A+tV}(\omega_t) }
    \le
    \frac{\pi}{2} \frac{\abs{t-s}\norm{V}}{d-t\norm{V_+}-t\norm{V_-}}
    ,
  \end{equation*}
  which proves the claim.
\end{proof}%

\subsection{The $\sin2\Theta$ theorem}\label{subsec:sin2Theta}

The following variant of the Davis-Kahan $\sin2\Theta$ theorem from~\cite{DK70} unifies~\cite[Theorem~1]{Seel14}
and~\cite[Proposition~2.4]{Seel19}.

\begin{proposition}\label{prop:sin2Theta}
  Let $A$ be as in Theorem~\ref{thm:gen}. Moreover, let $V$ be a bounded self-adjoint operator on the same Hilbert space and $Q$
  be an orthogonal projection onto a reducing subspace for $A+V$. Then, the operator angle
  $\Theta = \arcsin\abs{\EE_A(\sigma)-Q}$ associated with $\EE_A(\sigma)$ and $Q$ satisfies
  \begin{equation}\label{eq:sin2Theta}
    \norm{\sin2\Theta}
    \le
    \frac{\pi}{2} \frac{\norm{V_+}+\norm{V_-}}{d}
    .
  \end{equation}
  If, in addition, $\conv(\sigma) \cap \Sigma = \emptyset$ or $\sigma \cap \conv(\Sigma) = \emptyset$, then $\pi/2$
  in~\eqref{eq:sin2Theta} can be replaced by $1$.
\end{proposition}

\begin{proof}
  Recall from the proof of~\cite[Theorem~1]{Seel14} that
  \begin{equation*}
    \norm{\sin2\Theta}
    \le
    \frac{\pi}{2} \frac{\norm{V-KVK}}{d}
    ,
  \end{equation*}
  where $K = Q - Q^\perp$ is self-adjoint and unitary. Also recall from~\cite[Remark~2.5]{Seel14} that $\pi/2$ in this estimate
  can be replaced by $1$ if $\conv(\sigma) \cap \Sigma = \emptyset$ or $\sigma \cap \conv(\Sigma) = \emptyset$. It only remains
  to show that $\norm{V - KVK} \le \norm{V_+} + \norm{V_-}$.

  As in the proof of~\cite[Proposition~2.4]{Seel19}, the nonnegative operators $V_\pm$ satisfy
  $\norm{V_\pm - KV_\pm K} \le \norm{V_\pm}$. Hence, using $V = V_+ - V_-$, we have
  \begin{equation*}
    \norm{V - KVK}
    \le
    \norm{V_+ - KV_+K} + \norm{V_- - KV_-K}
    \le
    \norm{V_+} + \norm{V_-}
    ,
  \end{equation*}
  which completes the proof.
\end{proof}%

\begin{remark}\label{rem:sin2theta}
  With the~\emph{maximal angle} $\theta := \arcsin(\norm{ \EE_A(\sigma) - Q })$ between the subspaces $\Ran\EE_A(\sigma)$ and
  $\Ran Q$ and the inequality $\sin2\theta \le \norm{\sin2\Theta}$, we obtain from~\eqref{eq:sin2Theta} the~\emph{$\sin2\theta$
  estimate}
  \begin{equation}\label{eq:sin2theta}
    \sin2\theta
    \le
    \frac{\pi}{2} \frac{\norm{V_+} + \norm{V_-}}{d}
    ,
  \end{equation}
  where $\pi / 2$ can be replaced by $1$ if the additional separation condition mentioned in Proposition~\ref{prop:sin2Theta} is
  satisfied; cf.~\cite{Seel14,Seel19}. This estimate can also be derived by adapting the corresponding alternative reasoning
  from~\cite[Proposition~3.3]{Seel14} and~\cite[Proposition~A.1]{Seel19} (cf.~also~\cite[Corollary~4.3]{AM13}):
  
  Estimate~\eqref{eq:sin2theta} is clear if $\theta = \pi /2$, and for $\theta < \pi / 2$ we recall from the proof
  of~\cite[Proposition~3.3]{Seel14} that
  \begin{equation}\label{eq:sin2thetaproof}
    \sin2\theta
    \le
    \pi \frac{\norm{\EE_A(\Sigma) U^* V U \EE_A(\sigma)}}{d}
  \end{equation}
  with a certain unitary operator $U$ satisfying $U^*QU = \EE_A(\sigma)$. Also recall (see, e.g.,~\cite[Remark~3.2]{Seel14}) that
  the constant $\pi$ here can be replaced by $2$ if $\conv(\sigma) \cap \Sigma = \emptyset$ or
  $\sigma \cap \conv(\Sigma) = \emptyset$. Since  $U^* V_\pm U$ are nonnegative, by~\cite[Lemma~A.2]{Seel19} we have
  $2\norm{\EE_A(\Sigma) U^* V_\pm U \EE_A(\sigma)} \le \norm{U^* V_\pm U} = \norm{V_\pm}$. Thus,
  \begin{align*}
    2\norm{\EE_A(\Sigma) U^* V U \EE_A(\sigma)}
    &\le
    2\norm{\EE_A(\Sigma) U^* V_+ U \EE_A(\sigma)} + 2\norm{\EE_A(\Sigma) U^* V_- U \EE_A(\sigma)}\\
    &\le
    \norm{V_+} + \norm{V_-}
    ,
  \end{align*}
  which together with~\eqref{eq:sin2thetaproof} proves~\eqref{eq:sin2theta}.
\end{remark}

\subsection{Proof of the main results}\label{subsec:mainproofs}

\begin{proof}[Proof of Theorem~\ref{thm:favGeom}]
  Following the proof of~\cite[Theorem~1.1]{Seel19}, we see that estimate~\eqref{eq:maxAnglefavGeom} is a direct consequence of
  Proposition~\ref{prop:sin2Theta} with $Q = \EE_{A+V}(w)$ and Lemma~\ref{lem:projCont}.

  It remains to show the sharpness of estimate~\eqref{eq:maxAnglefavGeom}. This can be seen from the following example of
  $2 \times 2$ matrices (cf.~\cite[Remark~2.9]{Seel14} and the proof of~\cite[Theorem~1.1]{Seel19}): for arbitrary
  $0 \le v_\pm < 1$ with $v := v_+ + v_- < 1$ consider
  \begin{equation*}
    A
    :=
    \begin{pmatrix}
      \frac{1}{2} & 0\\
      0 & -\frac{1}{2}
    \end{pmatrix}
    \quad\text{ and }\
    V
    :=
    \begin{pmatrix}
      \frac{v_+ - v_- - v^2}{2} & \frac{v\sqrt{1-v^2}}{2}\\[0.1cm]
      \frac{v\sqrt{1-v^2}}{2} & \frac{v^2 + v_+ - v_-}{2}
    \end{pmatrix}
  \end{equation*}
  with $\sigma := \{ 1/2 \}$, $\Sigma := \{ -1/2 \}$, and $d := \dist(\sigma,\Sigma) = 1$.

  It is easy to verify that $\spec(V) = \{ -v_- , v_+ \}$ and that the spectrum of $A+V$ is given by
  $\spec(A+V) = \{ (v_+-v_- \pm \sqrt{1-v^2}) / 2 \}$. Denote
  $\omega := \{ (v_+-v_- + \sqrt{1-v^2}) / 2 \} \subset [1/2 - v_- , 1/2 + v_+]$ and $\theta := \arcsin(v) / 2$. Using the
  identities
  \begin{equation*}
    \frac{1-\sqrt{1-v^2}}{v}
    =
    \tan\theta
    =
    \frac{v}{1+\sqrt{1-v^2}}
    ,
  \end{equation*}
  it is then straightforward to show that
  \begin{equation*}
    (A+V) \begin{pmatrix} \cos\theta\\ \sin\theta \end{pmatrix}
    =
    \frac{v_+-v_- + \sqrt{1-v^2}}{2} \begin{pmatrix} \cos\theta\\ \sin\theta \end{pmatrix}
  \end{equation*}
  and, therefore,
  \begin{equation*}
    \arcsin(\norm{\EE_A(\sigma) - \EE_{A+V}(\omega)})
    =
    \theta
    =
    \frac{1}{2} \arcsin\Bigl( \frac{v_+ + v_-}{d} \Bigr)
    .
  \end{equation*}
  Thus, estimate~\eqref{eq:maxAnglefavGeom} is sharp, which completes the proof.
\end{proof}%

Just as Theorem~\ref{thm:favGeom}, we also obtain from Proposition~\ref{prop:sin2Theta} and Lemma~\ref{lem:projCont} the follow
result, which applies in the generic situation where no additional spectral separation condition other than~\eqref{eq:specSep} is
assumed. It unifies~\cite[Corollary~2]{Seel14} and~\cite[Corollary~2.5]{Seel19} and plays a crucial role in the proof of
Theorem~\ref{thm:gen}, see below.

\begin{corollary}\label{cor:sin2Thetagen}
  In the situation of Theorem~\ref{thm:gen} one has
  \begin{equation*}
    \arcsin(\norm{\EE_A(\sigma) - \EE_{A+V}(\omega)})
    \le
    \frac{1}{2} \arcsin\Bigl( \frac{\pi}{2} \frac{\norm{V_+}+\norm{V_-}}{d} \Bigr)
    \le
    \frac{\pi}{4}
  \end{equation*}
  whenever $\norm{V_+} + \norm{V_-} \le 2d / \pi$.
\end{corollary}

\begin{proof}[Proof of Theorem~\ref{thm:gen}]
  Let $0 = t_0 \le \dots \le t_n = 1$, $n \in \NN$, be a finite partition of the interval $[0,1]$, and set
  \begin{equation*}
    \lambda_j
    :=
    \frac{(t_{j+1}-t_j)(\norm{V_+} + \norm{V_-})}{d - t_j(\norm{V_+} + \norm{V_-})}
    <
    1
    ,\quad
    j = 0, \dots, n-1
    .
  \end{equation*}
  Considering $A+t_{j+1}V = (A+t_jV) + (t_{j+1}-t_j)V$, we obtain with Corollary~\ref{cor:sin2Thetagen} as
  in~\cite{Seel16,Seel18,Seel19} (cf.~also~\cite{AM13}) that
  \begin{equation}\label{eq:boundOpt}
    \arcsin(\norm{\EE_A(\sigma) - \EE_{A+V}(\omega)})
    \le
    \frac{1}{2} \sum_{j=0}^{n-1} \arcsin\Bigl( \frac{\pi\lambda_j}{2} \Bigr)
    \ \text{ whenever }\,
    \lambda_j
    \le
    \frac{2}{\pi}
    .
  \end{equation}
  Moreover, with $1-t_{j+1}( \norm{V_+} + \norm{V_-} ) / d = (1 - \lambda_j) (1 - t_j ( \norm{V_+} + \norm{V_-} ) / d )$ for
  $j = 0, \dots, n-1$, $t_0 = 0$, and $t_n = 1$, we have
  \begin{equation}\label{eq:parameters}
    1 - \frac{\norm{V_+} + \norm{V_-} }{d}
    =
    \prod_{j=0}^{n-1} (1-\lambda_j)
    .
  \end{equation}
  Recalling from the proof of~\cite[Theorem~1.2]{Seel19} that the function $N$ satisfies
  \begin{equation*}
    N\Bigl( \frac{x}{2} \Bigr)
    =
    \inf\biggl\{
    \frac{1}{2} \sum_{j=0}^{n-1} \arcsin\Bigl( \frac{\pi\lambda_j}{2} \Bigr) \colon n \in \NN,\
    0 \le \lambda_j \le \frac{2}{\pi},\ \prod_{j=0}^{n-1} (1-\lambda_j) = 1 - x
    \biggr\}
    ,
  \end{equation*}
  for $0 \le x \le 2c_\text{crit}$, the claim now follows from~\eqref{eq:boundOpt} and~\eqref{eq:parameters}.
\end{proof}%

\begin{remark}
  Considering partitions of the interval $[0 , 1]$ with arbitrarily small mesh size, we obtain from~\eqref{eq:boundOpt}
  analogously to~\cite[Remark~2.2]{Seel18},~\cite[Remark~2.1]{Seel16} and~\cite[Remark~2.6]{Seel19} that
  \begin{align*}
    \arcsin(\norm{\EE_A(\sigma) - \EE_{A+V}(\omega)})
    &\le
    \frac{\pi}{4} \int_0^1 \frac{\norm{V_+}+\norm{V_-}}{d - t\norm{V_+} - t\norm{V_-}} \,\dd t\\
    &=
    \frac{\pi}{4} \log \frac{d}{d - \norm{V_+} - \norm{V_-}}
    ,
  \end{align*}
  and the latter is strictly less than $\pi / 2$ whenever
  \begin{equation*}
   \frac{\norm{V_+} + \norm{V_-}}{d}
   \le
   2 \frac{\sinh(1)}{\exp(1)}
   <
   2 c_\text{crit}.
  \end{equation*}
\end{remark}


\end{document}